\documentclass[12pt,reqno]{amsart}

\usepackage[dvips]{graphicx} 
\usepackage[arrow,matrix,curve]{xy}

\usepackage{amssymb, latexsym, amsmath, amscd, array, 
}

\newtheorem{theorem}{Theorem}[section]

\theoremstyle{definition}

\newtheorem{remark}[theorem]{Remark}

\numberwithin{equation}{section}
\numberwithin{figure}{section} 
\numberwithin{table}{section}

\DeclareMathOperator{\PD}{{\rm PD}}

\newcommand\st{{\rm st}}

\newcommand \cqfd{\unskip\kern 6pt\penalty 500
\raise -2pt\hbox{\vrule\vbox to10pt{\hrule width 4pt
\vfill\hrule}\vrule}\par}                 

\def\adots{\mathinner{\mkern2mu\raise1pt\hbox{.}
\mkern3mu\raise4pt\hbox{.}\mkern1mu\raise7pt\hbox{.}}}
\def\hfl#1{\frac{\buildrel{#1}}{{\hbox to 12mm{\rightarrowfill}}}}
  
\def \\R^n \times \R^n
\rightarrow \R{\mathop{\R^n \times \R^n
\rightarrow \R}}

\newcommand\N {{\mathbb N}}

 \newcommand\R {{\mathbb R}}
\newcommand\RR {{\mathbb R}} 
 \newcommand\Z {{\mathbb Z}}

\newcommand{\smat}[4]
{{\(\!\!\begin{array}{cc}{#1}\!&\!{#2}\\begin{equation*}-0.1cm]{#3}\!&\!{#4}\end{array}\!\!\)}}

\long\def\forget#1\forgotten{} %

\long\def\forgett#1\forgottent{} %

\def\circ{\mathchoice%
 {\mathrel{\raise 1pt\hbox{$\scriptstyle\mathchar"020E$}}}
 {\mathrel{\raise 1pt\hbox{$\scriptstyle\mathchar"020E$}}}
 {\mathrel{\raise 1pt\hbox{$\scriptscriptstyle\mathchar"020E$}}}
 {}
}

\newcommand{\nc}{\newcommand} \nc{\on}{\operatorname}

\nc{\df}{\on{\it df}}

\nc{\conf}{\on{conf}}

\nc{\spt}{\on{spt}}
\nc{\norm}[1]{\| #1 \|}
\nc{\parallelleer}{\norm{\ }} 
\nc{\parallelh}{\norm h} 
\nc{\parallelk}{\norm k} 
\nc{\parallelx}{\norm x} 
\nc{\parallelhrr}{\norm {h_\RR}} 
\nc{\parallelom}{\norm \omega} 
\nc{\parallelomij}{\norm {\omega_{i_j}}} 
\nc{\parallelomx}{\norm {\omega_{x}}} 
\nc{\parallelpi}{\norm \pi} 
\nc{\parallelalf}{\norm \alpha} 
\nc{\parallelalfs}{\norm {\alpha_s}} 
\nc{\parallelalfi}{\norm {\alpha_i}} 
\nc{\parallelalfij}{\norm {\alpha_{i_j}}} 
\nc{\parallelbeta}{\norm \beta} 
\nc{\parallelbetat}{\norm {\beta_t}} 
\nc{\parallelhcapalf}{\norm {h \cap \alpha}} 
\nc{\parallelPDralf}{\norm {\PD_\RR(\alpha)}} 
\nc{\strichleer}{| \  |}
\nc{\NN}{\mathbb N}

\nc{\rr}{\mbox{$\scriptstyle\mathbb R$}}

\nc{\dF}{{\it dF}} 

\nc{\DF}{{\it DF}} 

\nc{\ds}{{\it ds}} 

\nc{\dvol}{{\it dvol}}

\nc{\grad}{{\rm grad}} 
\nc{\strichw}{\|\omega\|} 
\nc{\strichwx}{|\omega_x|}
\nc{\Hess}{{\rm Hess}}

\begin{document}

\title{A strict non-standard inequality~$.999\ldots < 1$}

\author{Karin Usadi Katz}

\author[M.~Katz]{Mikhail G. Katz$^{*}$}

\address{Department of Mathematics, Bar Ilan University, Ramat Gan
52900 Israel} \email{\{katzmik\}@macs.biu.ac.il (remove curly braces)}

\thanks{$^{*}$Supported by the Israel Science Foundation (grants
no.~84/03 and 1294/06) and the BSF (grant 2006393)}

\subjclass[2000]{Primary 26E35; 
Secondary 97A20,
97C30
}

\keywords{decimal representation, generic limit, hyperinteger,
infinitesimal, Lightstone's semicolon, non-standard calculus, unital
evaluation}

\date{\today}

\begin{abstract}
Is~$.999\ldots$ equal to~$1$?  A.~Lightstone's decimal expansions
yield an infinity of numbers in~$[0,1]$ whose expansion starts with an
unbounded number of repeated digits ``$9$''.  We present some
non-standard thoughts on the ambiguity of the ellipsis, modeling the
cognitive concept of {\em generic limit\/} of B.~Cornu and D.~Tall.  A
choice of a non-standard hyperinteger $H$ specifies an $H$-infinite
extended decimal string of~$9$s, corresponding to an infinitesimally
diminished hyperreal value~\eqref{105b}.  In our model, the student
resistance to the unital evaluation of $.999\ldots$ is directed
against an unspoken and unacknowledged application of the standard
part function, namely the stripping away of a ghost of an
infinitesimal, to echo George Berkeley.  So long as the number system
has not been specified, the students' hunch that $.999\ldots$ can fall
infinitesimally short of $1$, can be justified in a mathematically
rigorous fashion.
\end{abstract}

\maketitle 

\tableofcontents

\section{The problem of unital evaluation}

Student resistance to the evaluation of~$.999\ldots$ as the real
number~$1$ (henceforth referred to as the {\em unital evaluation}) has
been widely discussed in the mathematics education literature.  It has
been suggested that the source of such resistance lies in a
psychological predisposition in favor of thinking of~$.999\ldots$ as a
process, or iterated procedure, rather than the final outcome, see for
instance D.~Tall's papers \cite[p.~6]{TS}, \cite[p.~221]{Ta00},
\cite{Ta80} (see also \cite{Ta09b} for another approach).

We propose an alternative model to explain such resistance, in the
framework of non-standard analysis.  From this point of view, the
resistance is directed against an unspoken and unacknowledged
application of the standard part function ``st'' (see
Section~\ref{ten}, item \ref{103}), namely the stripping away of a
ghost of an infinitesimal, to echo George Berkeley~\cite{Be}, implicit
in unital evaluation:

\newcommand{\tail}{{\includegraphics[height=.6in]{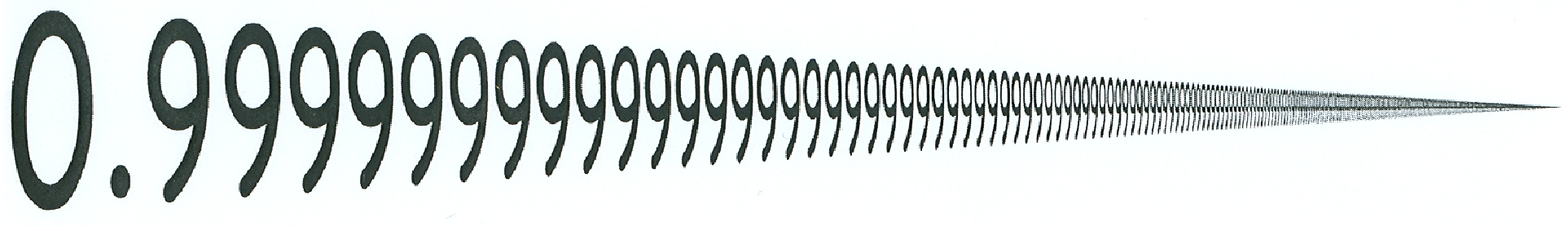}}}

\begin{figure}[ht]
\[
\st \begin{pmatrix}\tail\end{pmatrix} =1.
\]
\caption{Taking standard part of wikiartist's conception of an
infinity of 9s}
\label{9tail}
\end{figure}

What is the significance of such fine (indeed, infinitesimal)
districtions between $.999\ldots$ and $1$?  As this text is addressed
to a wide audience, it may be worth recalling that {\em dividing\/} by
the difference $1-.999\ldots$ will only be possible if the latter is
nonzero, a matter of immediate import for studying rates of change and
calculating derivatives (see Section~\ref{zoom}).

The hyperreal approach to a certain extent vindicates the student
resistance to education professionals' toeing the standard line on
unital evaluation: so long as the number system has not been
specified, the students' hunch that $.999\ldots$ can fall
infinitesimally short of $1$, can be justified in a mathematically
rigorous fashion.  

In Section~\ref{two}, based on a formula for a geometric sum, we
summarize the standard argument in favor of unital evaluation.  In
Section~\ref{three}, we summarize the standard resistance to the
latter.  In Section~\ref{four}, we construct a hyperreal decimal in
$[0,1)$, represented by a string of $H$-infinitely many repeated 9s,
and identify its Lightstone representation.  In Section~\ref{five}, we
square the strict inequality with standard reality, by means of the
standard part function.  In Section~\ref{six}, we represent the
hyperreal graphically by means of an infinite-magnification
microscope, already exploited for pedagogical purposes by Keisler
\cite{Ke} and Tall \cite{Ta80}.  In Section~\ref{zoom}, we exploit the
hyperreal to calculate $f'(1)$.

In Section~\ref{eight}, we develop an applied-mathematical model of a
hypercalculator so as to explain a familiar phenomenon of a calculator
returning a string of 9s in place of an integer.  In
Section~\ref{nine}, we examine the cognitive concept of {\em generic
limit\/} of Cornu and Tall in mathematics education, in relation to a
hyperreal approach to limits.  In Section~\ref{flat}, we set the
situation in perspective, with E. A. Abbott.

Section~\ref{ten} is a technical appendix containing basic material on
non-standard calculus.  The historical Section~\ref{history} contains
an examination of the views of Courant, Lakatos, and E. Bishop.
Section~\ref{proposal} contains a 10-step proposal concerning the
problem of unital evaluation.

\section{A geometric sum}
\label{two}

Evaluating the formula
\begin{equation*}
1+r+r^2+\ldots+r^{n-1}=\frac{1-r^{n}}{1-r}
\end{equation*}
at~$r=\frac{1}{10}$, we obtain
\begin{equation*}
1+\frac{1}{10}+\frac{1}{100}+\ldots + \frac{1}{10^{n-1}}=
\frac{1-\frac{1}{10^{n}}}{1-\frac{1}{10}} ,
\end{equation*}
or alternatively

\begin{equation*}
\underset{n}{\underbrace{1.11\ldots 1}} =
\frac{1-\frac{1}{10^{n}}}{1-\frac{1}{10}} .
\end{equation*}
Multiplying by~$\frac{9}{10}$, we obtain
\begin{equation*}
\begin{aligned}
.\underset{n}{\underbrace{999\ldots 9}}\; & = \frac{9}{10} \left(
\frac{1-\frac{1}{10^{n}}}{1-\frac{1}{10}} \right) \\&=1 -
\frac{1}{10^{n}} 
\end{aligned}
\end{equation*}
for every~$n\in \N$.  As~$n$ increases without bound, the formula
\begin{equation}
\label{21}
.\underset{n}{\underbrace{999\ldots 9}}\; = 1- \frac{1}{10^{n}}
\end{equation}
becomes
\begin{equation*}
.999\ldots = 1.
\end{equation*}
Or does it?

\section{Arguing by ``I told you so''}
\label{three}

When I tried this one on my teenage daughter, she remained
unconvinced.  She felt that the number~$.999\ldots$ is smaller
than~$1$.  After all, just look at it!  There is something missing
before you reach~$1$.  I then proceeded to give a number of arguments.
Apologetic mumbo-jumbo about the alleged ``non-unicity of decimal
representation'' fell on deaf ears.  The one that seemed to work best
was the following variety of the old-fashioned ``because I told you
so'' argument: factor out a~$3$:
\begin{equation*}
3(.333\ldots) = 1
\end{equation*}
to obtain
\begin{equation}
\label{33}
.333\ldots = \frac{1}{3}
\end{equation}
and ``everybody knows'' that the number~$.333\ldots$ is exactly ``a
third'' {\em on the nose}.  Q.E.D.  This worked for a few minutes, but
then the validity of \eqref{33} was called into question, as well.

\section{Coming clean}
\label{four}

Then I finally broke down.  In Abraham Robinson's theory of hyperreal
analysis \cite{Ro}, there is a notion of an infinite hyperinteger (see
Section~\ref{ten}, item~\ref{108}).  H.~Jerome Keisler \cite{Ke} took
to denoting such an entity by the symbol
\begin{equation*}
H,
\end{equation*}
possibly because its inverse is an infinitesimal~$h$ typically
appearing in the denominator of the familiar definition of derivative
(it most decidedly does {\em not\/} stand for ``Howard'').  Taking
infinitely many terms in formula \eqref{21} is interpreted as
replacing~$n\in \N$ by an infinite hyperinteger
\begin{equation*}
H \in \N^* \setminus \N.
\end{equation*}
The transfer principle (see Section~\ref{ten}, item~\ref{101}) applied
to \eqref{21} then yields
\begin{equation*}
.\underset{H}{\underbrace{999\ldots}}\; = 1- \frac{1}{10^{H}}
\end{equation*}
where the infinitesimal quantity~$\frac{1}{10^{H}}$ is nonzero:
\begin{equation*}
\frac{1}{10^{H}} > 0.
\end{equation*}
Therefore we obtain the strict nonstandard inequality
\begin{equation}
\label{41}
.\underset{H}{\underbrace{999\ldots}}\; < 1
\end{equation}
and my teenager was right all along.  Note that hyperreal extended
decimal expansions were treated by A. Lightstone in
\cite[pp.~245--247]{Li}.  In his notation, the hyperreal appearing in
the left-hand-side of~\eqref{41} appears as the extended decimal
\[
.999\ldots;\ldots 99\hat 9
\]
with the hat ``$\hat{\phantom{9}}$'' indicating the~$H$-th decimal
place, where the {\em last\/} repeated digit~$9$ occurs.  We have
employed the underbrace notation as in \eqref{41} rather than
Lightstone's semicolon, as it parallels the finite case more closely,
and seems more intuitive.  An alternative construction of a strict
inequality $.999\ldots<1$ may be found in \cite{Ri}, however in a
number system which is not a field.

\section{Squaring~$.999\ldots <1$ with reality}
\label{five}

To obtain a real number in place of the hyperreal
$.\underset{H}{\underbrace{999\ldots}}$, we apply the standard part
function ``$\st$'' (see Section~\ref{ten}, item~\ref{103}):
\begin{equation*}
\st \left(.\underset{H}{\underbrace{999\ldots}} \right) = \st \left(
1- \frac{1}{10^{H}} \right) = 1 -\st \left( \frac{1}{10^{H}} \right) =
1.
\end{equation*}
To elaborate further, one could make the following remark.  Even in
standard analysis, the expression~$.999\ldots$ is only shorthand for
the {\bf limit} of the finite expression \eqref{21} when the standard
integer~$n$ increases without bound.  From the hyperreal viewpoint,
``taking the limit'' is accomplished in two steps:
\begin{enumerate}
\item
evaluating the expression at an infinite hyperinteger~$H$, and then
\item
taking the standard part.
\end{enumerate}

The two steps form the non-standard definition of limit (see
Section~\ref{ten}, item~\ref{1010}).  Now the first step (evaluating
at~$H$) produces a hyperreal number dependent on~$H$ (in all cases it
will be strictly less than~$1$).  The second step will strip away the
infinitesimal part and produce the standard real number~$1$ in the
{\em cluster\/} (see Section~\ref{ten}, item~\ref{104}) of points
infinitely close to it.

\begin{remark}[Multiple infinities]
\label{51b}
The existence of more than one infinite hyperreal is not only a
requirement to have a field, but is actually extremely useful.  For
example, using the natural hypperreal extension~$f^*$ of~$f$ (see
Section~\ref{ten}, item~\ref{101}), it is possible to write down a
pointwise definition of uniform continuity of a function~$f$ (see
below).  Such a definition considerably reduces the quantifier
complexity of the standard definition.

To elaborate, note that the standard definition of uniform continuity
of a real function~$f$ can be said to be {\em global\/} rather than
{\em local\/} (i.e. pointwise), in the sense that, unlike ordinary
continuity, uniform continuity cannot be defined as a pointwise
property of~$f$.  Meanwhile, in the framework of Robinson's theory, it
is possible to give a definition of uniform continuity of the real
function~$f$ in terms of its natural hyperreal extension, denoted
\begin{equation*}
f^*,
\end{equation*}
in such a way that the definition is local, in the sense of depending
only on each pointwise cluster (see Section~\ref{ten}, item~\ref{104})
in the domain of~$f^*$.  Thus,~$f$ is uniformly continuous on~$\R$ if
the following condition is satisfied:
\begin{equation*}
\forall x\in \R^* \quad \left( y\approx x \implies
f^{*\phantom{i\!\!}}_{\phantom{I}}(y)\approx f^*(x)\right) .
\end{equation*}
Here~$\approx$ stands for the relation of being infinitely close.  The
condition must be satisfied at the infinite hyperreals~$x$, in
addition to the finite ones.  This addition is what distinguishes
uniform continuity from ordinary continuity.
\end{remark}

\section{Hyperreals under magnifying glass}
\label{six}

The symbol ``$\infty$'' is employed in standard real analysis to
define a formal completion of the real line~$\R$, namely
\begin{equation}
\label{51}
\R\cup \{\infty\}
\end{equation}
(sometimes a formal point~``$-\infty$'' is added, as well).  

\begin{figure}
\includegraphics[height=4in]{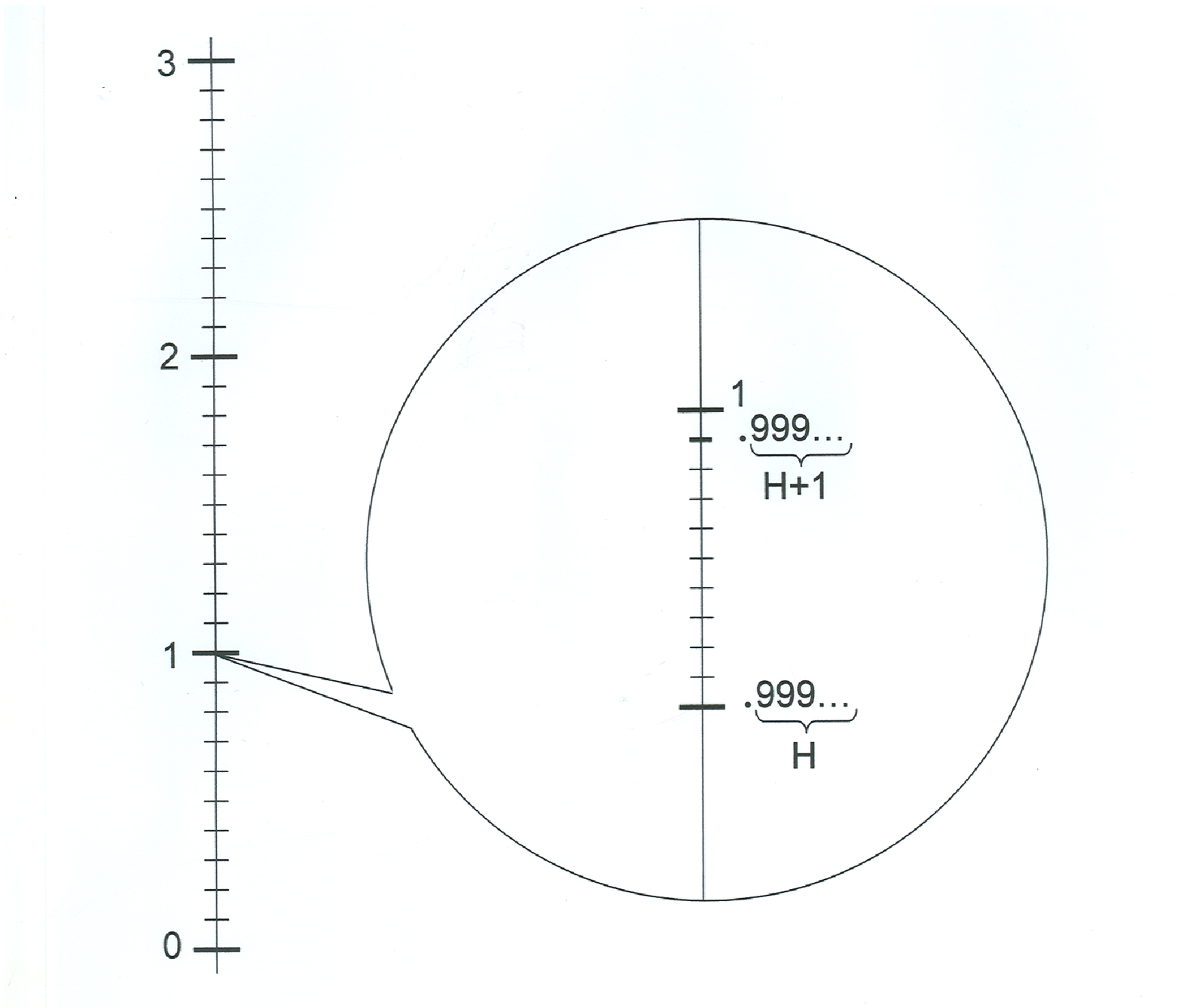}
\caption{Three infinitely close hyperreals under a microscope}
\label{micro}
\end{figure}

Such a formal device is helpful in simplifying the statements of
certain theorems (which would otherwise have a number of subcases).
The symbol is used in a {\em different\/} sense in topology and
projective geometry, where the Thom compactification~$\R\cup
\{\infty\}$ of~$\R$ is a circle:
\begin{equation}
\label{52b}
\R\cup \{\infty\} \approx S^1.
\end{equation}
We have refrained from using the symbol ``$\infty$'' to denote an
infinite hyperreal, as in
\[
.\underset{\infty}{\underbrace{999\ldots}},
\]
even though the symbol ``$\infty$'' does convey the idea of the
infinite more effectively than the symbol ``$H$'' that we have used.
The reason is so as to avoid the risk of creating a false impression
of the uniqueness of an infinite point (as in \eqref{51} or
\eqref{52b} above), in the field~$\R^*$ (see Section~\ref{ten},
item~\ref{102}).

We represent the hyperreal
\begin{equation}
\label{52}
.\underset{H}{\underbrace{999\ldots}}
\end{equation}
visually by means of an infinite-resolution microscope already
exploited for pedagogical purposes by Keisler \cite{Ke}.  The
hyperreal~$.\underset{H}{\underbrace{999\ldots}}$ appearing in the
diagram of Figure~\ref{micro} illustrates graphically the strict
hyperreal inequality
\begin{equation*}
.\underset{H}{\underbrace{999\ldots}} < 1,
\end{equation*}
where, as we already mentioned, the symbol~$H$ is exploited to denote
a fixed infinite Robinson hyperinteger (see Section~\ref{ten},
item~\ref{108}).

\section{Zooming in on slope of tangent line}
\label{zoom}

To calculate the slope of the tangent line to the curve $y=x^2$ at
$x=1$, we first compute the ratio
\[
\frac{\Delta y}{\Delta x} = \frac{(.9..)^2 -1^2}{.9.. - 1} =
\frac{(.9.. -1)(.9.. + 1)}{.9.. -1} = .9.. + 1,
\]
where we have deleted the underbrace
$\underset{H}{\underbrace{\phantom{.999\ldots}}}$ and also shortened
the symbol~$.999\ldots$ to $.9..$, so as to lighten the notation.

\begin{figure}
\includegraphics[height=3in]{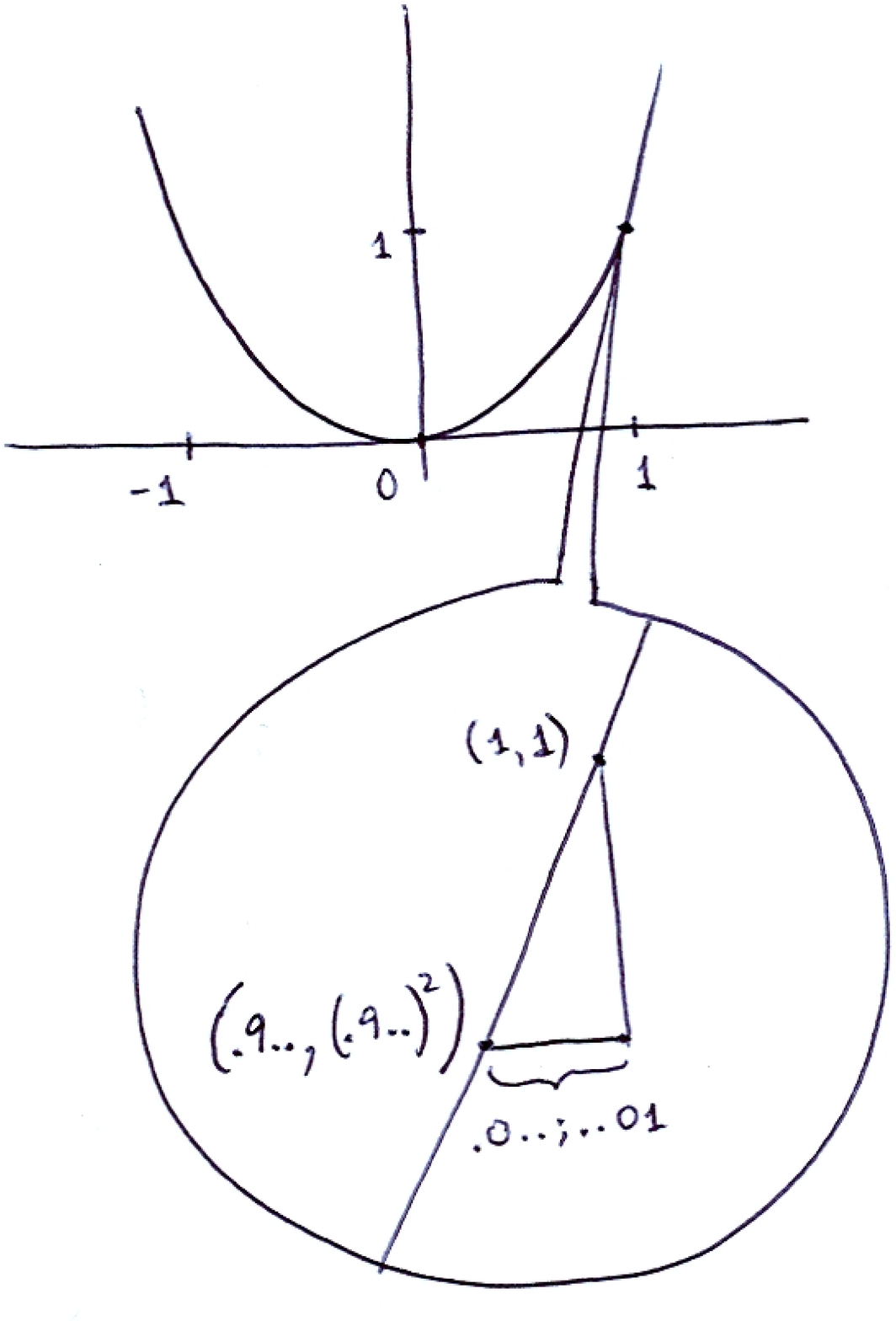}
\caption{Zooming in on the slope of tangent line to the curve $y=x^2$
at $x=1$}
\label{slope}
\end{figure}

Therefore the slope is
\[
\frac{dy}{dx} = \st\left( \frac{\Delta y}{\Delta x} \right) =
\st(.9.. + 1) = \st(.9..) + 1 = 1+1=2.
\]
Note that 
\[
\Delta x = .9.. - 1 = -.0..;..01
\]
in Lightstone's notation, where the digit ``1'' appears in the $H^{\rm
th}$-decimal place, as illustrated in Figure~\ref{slope}.

\section{Hypercalculator returns~$.999\ldots$}
\label{eight}

Everyone who has ever held an electronic calculator is familiar with
the curious phenomenon of it sometimes returning the value
\begin{equation*}
.999999
\end{equation*}
in place of the expected~$1.000000$.  For instance, a calculator
programmed to apply Newton's method to find the zero of a function,
may return the~$.999999$ value as the unique zero of the function
\[
\log x.
\]
Developing a model to account for such a phenomenon is complicated by
the variety of the degree of precision displayed, as well as the
greater precision typically available internally than that displayed
on the LCD.  To simplify matters, we will consider an idealized model,
called a {\em hypercalculator}, of a theoretical calculator that
applies Newton's method precisely~$H$ times, where~$H$ is a fixed
infinite hyperinteger, as discussed in the previous section.

\begin{theorem}
Let~$f$ be a concave increasing differentiable function with domain an
open interval~$(1-\epsilon, 1+\epsilon)$ and vanishing at its
midpoint.  Then the hypercalculator applied to~$f$ will return a
hyperreal decimal~$.999\ldots$ with an initial segment consisting of
an unbounded number of repeated~$9$'s.
\end{theorem}

\begin{proof}
Assume for simplicity that~$f(x_0)<0$.  We have 
\begin{equation*}
x_1= x_0 + \frac{|f(x_0)|}{f'(x_0)}.
\end{equation*}
By the mean value theorem, there is a point~$c$ such that~$ x_0 < c <
1$ where~$f'(c)= \frac{|f(x_0)|}{1-x_0}$, or
\begin{equation*}
\frac{|f(x_0)|}{f'(c)}= 1-x_0 .
\end{equation*}
Since~$f$ is concave, its derivative~$f'$ is decreasing, hence
\begin{equation*}
x_1 = x_0 + \frac{|f(x_0)|}{f'(x_0)} < x_0 +1 -x_0 = 1 .
\end{equation*}
Thus~$x_1 < 1$.  Inductively, the point~$x_{n+1}= x_n +
\frac{|f(x_n)|} {f'(x_n)}$ satisfies~$x_{n}<1$ for all~$n$.  By the
transfer principle (see Section~\ref{ten}, item~\ref{101}), the
hyperreal~$x_{H}$ satisfies a strict inequality
\begin{equation*}
x_{H}<1,
\end{equation*}
as well (cf.~\eqref{41}).  Hence the hypercalculator returns a value
strictly smaller than~$1$ yet infinitely close to~$1$, proving the
theorem.
\end{proof}

\section{{\em Generic limit\/} and precise meaning of infinity}
\label{nine}

The precise meaning of the finite expression
\begin{equation*} 
.999...9, \quad n \; {\rm times}
\end{equation*}
is that the repeated digit 9 occurs precisely~$n$ times.  The standard
non-terminating decimal
\begin{equation*}
.999\ldots,
\end{equation*}
as it is traditionally written, is said to have an unbounded number of
repeated digits 9, but the expression ``infinitely many~9's" is only a
figure of speech, as ``infinity" is not a number in standard analysis,
in the sense that, whenever a precise meaning is attributed to the
phrase ``infinitely many~9's", it is almost invariably in terms of
{\em limits}.  D. Tall writes in \cite{Ta09a} as follows:
\begin{quote}
[...] the infinite decimal~$0.999\ldots$ is intended to signal the
{\em limit\/} of the sequence~$0.9,\;$~$0.99,\;$~$0.999$, ... which is
$1$, but in practice it is ofter imagined as a limiting process which
never quite reaches~$1$.
\end{quote}
Tall \cite{Ta91, Ta09a} describes a concept in cognitive theory he
calls a {\em generic limit concept\/} in the following terms:
\begin{quote}
[...] if a quantity repeatedly gets smaller and smaller and smaller
without ever being zero, then the limiting object is naturally
conceptualised as an extremely small quantity that is not zero (Cornu
\cite{Co91}).  Infinitesimal quantities are natural products of the
human imagination derived through combining potentially infinite
repetition and the recognition of its repeating properties.
\end{quote}
(see \cite{GT} for the related notion of {\em procept}).  A technical
realisation of the cognitive concept of {\em generic limit\/} is
achieved in the hyperreal line (see Section~\ref{ten}, item~\ref{102})
as follows.  Given an infinite hyperinteger~$H$, consider the
hyperreal repeated decimal where the repeated digit~$9$ occurs
precisely~$H$ times, in the sense routinely used in non-standard
calculus, for example, when one partitions a compact interval into~$H$
parts in the proof of the extreme value theorem (see
Section~\ref{ten}, item~\ref{109}).  Such a number can be denoted
suggestively by~$.\underset{H}{\underbrace{999\ldots}}$ with an
underbrace indicating that the digit~$9$ occurs~$H$ times, resulting
in a strict inequality~$.\underset{H}{\underbrace{999\ldots}} < 1$
(with the underbrace indicating that we are not referring to the
standard real).  In Lightstone's notation, this hyperreal would be
expressed by the hyperdecimal
\begin{equation*}
.999...;...9, 
\end{equation*}
the last digit 9 occurring in the~$H$-th position.  Such a hyperreal
appears to be the mathematical counterpart of the cognitive concept of
{\em generic limit}.

\section{Limits, generic limits, and Flatland}
\label{flat}

As far as standard limits are concerned, given the sequence~$u_1=.9$,
$u_2=.99$,~$u_3=.999$, etc., from the hyperreal viewpoint we have
\begin{equation*}
\lim_{n\to \infty} u_n = {\rm st}(u_H),
\end{equation*}
where ``st" is the standard part function which ``strips away" the
infinitesimal part, and outputs the standard real in the cluster of
the hyperreal~$u_H$ (see Section~\ref{ten}, item~\ref{104}).  What may
be bothering the students is this unacknowledged application of the
standard part function, resulting in a loss of an infinitesimal.

We have, in fact, been looking at the problem ``from above'', in the
context of non-standard analysis.  Perhaps a helpful parallel is
provided by the famous animated film {\em Flatland} (cf.~\cite{Ab}),
where the two--dimensional creatures are unable to conceive of what we
think of as the sphere in 3-space, due to their dimension limitation.
Similarly, one can conceive of the difficulty in the understanding of
the unital evaluation of~$.999\ldots$, as due to the limitation of the
standard real vision.

The notion of an infinitesimal appeals to intuition (see R.~Courant's
comment in Section~\ref{history}) and would not go away inspite of
what is, by now, over a century of~$(\epsilon, \delta)$-ideology (see
E.~Bishop's comment in Section~\ref{history}).  Highschool students
are exposed to the thorny~$.999\ldots$ issue before they are exposed
to any rigorous notion of a {\em real number}.  They are not aware of
fine differences between rational numbers, algebraic numbers, real
numbers, hyperreal numbers.  A related point is make by Keisler in his
textbook \cite{Ke}, when he points out that ``we have no way of
knowing what the line in physical space looks like''.  Most students
(perhaps all) initially believe that the mysterious number with
``infinitely many'' repeated digits 9 falls short of the value~$1$.
If an education professional claims that the students are making a
mistake, might he in fact be making a pedagogical error?

\section{A non-standard glossary}
\label{ten}

In this section we present some illustrative terms and facts from
non-standard calculus \cite{Ke}.  The relation of being infinitely
close is denoted by the symbol~$\approx$.  Thus,~$x\approx y$ if and
only if~$x-y$ is infinitesimal.

\subsection{Natural hyperreal extension~$f^*$}
\label{101}
The {\em extension principle\/} of non-standard calculus states that
every real function~$f$ has a hyperreal extension, denoted~$f^*$ and
called the natural extension of~$f$.  The {\em transfer principle\/}
of non-standard calculus asserts that every real statement true
for~$f$, is true also for~$f^*$.  For example, if~$f(x)>0$ for every
real~$x$ in its domain~$I$, then~$f^*(x)>0$ for every hyperreal~$x$ in
its domain~$I^*$.  Note that if the interval~$I$ is unbounded,
then~$I^*$ necessarily contains infinite hyperreals.  We will
typically drop the star~$^*$ so as not to overburden the notation.

\subsection{Internal set}
\label{102}
Internal set is the key tool in formulating the transfer principle,
which concerns the logical relation between the properties of the real
numbers~$\R$, and the properties of a larger field denoted
\[
\R^*
\]
called the {\em hyperreal line}.  The field~$\R^*$ includes, in
particular, infinitesimal (``infinitely small") numbers, providing a
rigorous mathematical realisation of a project initiated by Leibniz.
Roughly speaking, the idea is to express analysis over~$\R$ in a
suitable language of mathematical logic, and then point out that this
language applies equally well to~$\R^*$.  This turns out to be
possible because at the set-theoretic level, the propositions in such
a language are interpreted to apply only to internal sets rather than
to all sets.  Note that the term ``language" is used in a loose sense
in the above.  A more precise term is {\em theory in first-order
logic}.  Internal sets include natural extension of standard sets.

\subsection{Standard part function}
\label{103}
The standard part function ``st" is the key ingredient in Abraham
Robinson's resolution of the paradox of Leibniz's definition of the
derivative as the ratio of two infinitesimals
\[
\frac{dy}{dx}.
\]
The standard part function associates to a finite hyperreal
number~$x$, the standard real~$x_0$ infinitely close to it, so that we
can write
\begin{equation*}
\mathrm{st}(x)=x_0.
\end{equation*}
In other words, ``st'' strips away the infinitesimal part to produce
the standard real in the cluster.  The standard part function ``st" is
not defined by an internal set (see item~\ref{102} above) in
Robinson's theory.

\subsection{Cluster}
\label{104}
Each standard real is accompanied by a cluster of hyperreals
infinitely close to it.  The standard part function collapses the
entire cluster back to the standard real contained in it.  The cluster
of the real number~$0$ consists precisely of all the infinitesimals.
Every infinite hyperreal decomposes as a triple sum
\[
H+r+\epsilon,
\]
where~$H$ is a hyperinteger,~$r$ is a real number in~$[0,1)$, and
$\epsilon$ is infinitesimal.  Varying~$\epsilon$ over all
infinitesimals, one obtains the cluster of~$H+r$.

\subsection{Derivative}
\label{105}
To define the real derivative of a real function~$f$ in this approach,
one no longer needs an infinite limiting process as in standard
calculus.  Instead, one sets
\begin{equation}
\label{deri}
f'(x) = \mathrm{st} \left( \frac{f(x+\epsilon)-f(x)}{\epsilon}
\right),
\end{equation}
where~$\epsilon$ is infinitesimal, yielding the standard real number
in the cluster of the hyperreal argument of ``st'' (the derivative
exists if and only if the value~\eqref{deri} is independent of the
choice of the infinitesimal).  The addition of ``st'' to the formula
resolves the centuries-old paradox famously criticized by George
Berkeley \cite{Be} (in terms of the {\em Ghosts of departed
quantities}, cf.~\cite[Chapter~6]{St}), and provides a rigorous basis
for the calculus.

\subsection{Continuity}
\label{106}
A function~$f$ is continuous at~$x$ if the following condition is
satisfied:~$y\approx x$ implies~$f(y)\approx f(x)$.

\subsection{Uniform continuity} 
\label{107}
A function~$f$ is uniformly continuous on~$I$ if the following
condition is satisfied:

\begin{itemize}
\item
standard: for every~$\epsilon>0$ there exists a~$\delta>0$ such that
for all~$x\in I$ and for all~$y\in I$, if~$|x-y|<\delta$ then
$|f(x)-f(y)| < \epsilon$.
\item
non-standard: for all~$x\in I^*$, if~$x\approx y$ then~$f(x) \approx
f(y)$.
\end{itemize}
See Remark~\ref{51b} for a more detailed discussion.

\subsection{Hyperinteger}
\label{108}
A hyperreal number~$H$ equal to its own integer part 
\[
H = [H]
\]
is called a hyperinteger (here the integer part function is the
natural extension of the real one).  The elements of the complement
$\Z^* \setminus \Z$ are called infinite hyperintegers.

\subsection{Proof of extreme value theorem}
\label{109}
Let~$H$ be an infinite hyperinteger.  The interval~$[0,1]$ has a
natural hyperreal extension.  Consider its partition into~$H$
subintervals of equal length~$\frac{1}{H}$, with partition points~$x_i
= i/H$ as~$i$ runs from~$0$ to~$H$.  Note that in the standard
setting, with~$n$ in place of~$H$, a point with the maximal value
of~$f$ can always be chosen among the~$n+1$ partition points~$x_i$, by
induction.  Hence, by the transfer principle, there is a
hyperinteger~$i_0$ such that~$0\leq i_0 \leq H$ and
\begin{equation}
\label{101b}
f(x_{i_0})\geq f(x_i) \quad \forall i= 0,...,H.
\end{equation}
Consider the real point
\begin{equation*}
c= {\rm st}(x_{i_0}).
\end{equation*}
An arbitrary real point~$x$ lies in a suitable sub-interval of the
partition, namely~$x\in [x_{i-1},x_i]$, so that~${\rm st}(x_i) = x$.
Applying ``st'' to the inequality \eqref{101b}, we obtain by continuity
of~$f$ that~$f(c)\geq f(x)$, for all real~$x$, proving~$c$ to be a
maximum of~$f$ (see \cite[p.~164]{Ke}).

\subsection{Limit}
\label{1010}
We have~$\lim_{x\to a} f(x) = L$ if and only if whenever the
difference~$x-a$ is infinitesimal, the difference~$f(x)-L$ is
infinitesimal, as well, or in formulas: if~${\rm st}(x)=a$ then~${\rm
st}(f(x)) = L$.

Given a sequence of real numbers~$\{x_n|n\in \mathbb{N}\}$, if~$L\in
\mathbb{R}\;$ we say~$L$ is the limit of the sequence and write~$L =
\lim_{n \to \infty} x_n$ if the following condition is satisfied:
\begin{equation}
\label{102b}
{\rm st} (x_H)=L \quad \mbox{\rm for all infinite } H
\end{equation}
(here the extension principle is used to define~$x_n$ for every
infinite value of the index).  This definition has no quantifier
alternations.  The standard~$(\epsilon, \delta)$-definition of limit,
on the other hand, does have quantifier alternations:
\begin{equation}
\label{disaster}
L = \lim_{n \to \infty} x_n\Longleftrightarrow \forall \epsilon>0\;,
\exists N \in \mathbb{N}\;, \forall n \in \mathbb{N} : n >N \implies
d(x_n,L)<\epsilon.
\end{equation}          

\subsection{Non-terminating decimals}
\label{1011}
Given a real decimal 
\[
u=.d_1d_2d_3\ldots,
\]
consider the sequence~$u_1=.d_1$,~$\;u_2=.d_1 d_2$,~$\;u_3=.d_1 d_2
d_3$, etc.  Then by definition,
\[
u=\lim_{n\to \infty} u_n= \st(u_H^{\phantom{I}})
\]
for every infinite~$H$.  Now if~$u$ is a non-terminating decimal, then
one obtains a strict inequality~$u_H^{\phantom{I}}<u$ by transfer from
$u_n<u$.  In particular,
\begin{equation}
\label{105b}
.999\ldots;\ldots \hat 9 = .\underset{H}{\underbrace{999\ldots}} = 1 -
\tfrac{1}{10^H} < 1,
\end{equation}
where the hat~$\hat {\phantom{9}}$ indicates the~$H$-th Lightstone
decimal place.  The standard interpretation of the symbol~$.999\ldots$
as~$1$ is necessitated by notational uniformity: the symbol~$.a_1 a_2
a_3 \ldots$ in every case corresponds to the limit of the sequence of
terminating decimals~$.a_1\ldots a_n$.  Alternatively, the ellipsis
in~$.999\ldots$ could be interpreted as alluding to an infinity of
nonzero digits specified by a choice of an infinite hyperinteger~$H\in
\N^*\setminus \N$.  The resulting~$H$-infinite extended decimal string
of~$9$s corresponds to an infinitesimally diminished hyperreal
value~\eqref{105b}.  Such an interpretation is perhaps more in line
with the naive initial intuition persistently reported by teachers.

\section{Courant, Lakatos, and Bishop}
\label{history}

Prior to Robinson, mathematicians thought of infinitesimals in terms
of ``naive befogging" and ``vague mystical ideas". Thus, Richard
Courant \cite[p.~81]{Cou37} wrote as follows:
\begin{quote}
We must, however, guard ourselves against thinking of~$dx$ as an
``infinitely small quantity" or ``infinitesimal", or of the integral
as the ``sum of an infinite number of infinitesimally small
quantities". Such a conception would be devoid of any clear meaning;
it is only a naive befogging of what we have previously carried out
with precision.
\end{quote}
and again on page 101:
\begin{quote}
We have no right to suppose that first~$\Delta x$ goes through
something like a limiting process and reaches a value which is
infinitesimally small but still not 0, so that~$\Delta x$ and~$\Delta
y$ are replaced by ``infinitely small quantities" or ``infinitesimals"
$dx$ and~$dy$, and that the quotient of these quantities is then
formed. Such a conception of the derivative is incompatible with the
clarity of ideas demanded in mathematics; in fact, it is entirely
meaningless. For a great many simple-minded people it undoubtedly has
a certain charm, the charm of mystery which is always associated with
the word ``infinite"; and in the early days of the differential
calculus even Leibnitz [sic] himself was capable of combining these
vague mystical ideas with a thoroughly clear understanding of the
limiting process. It is true that this fog which hung round the
foundations of the new science did not prevent Leibnitz [sic] or his
great successors from finding the right path. But this does not
release us from the {\bf duty of avoiding every such hazy idea}
[emphasis added--MK] in our building-up of the differential and
integral calculus.
\end{quote}
Needless to say, Courant's criticism was not without merit {\em at the
time of its writing\/} (a quarter century prior to Robinson's
discovery).  I.~Lakatos \cite[p.~44]{La} wrote in '66 as follows:
\begin{quote}
Robinson's work... offers a rational reconstruction of the discredited
infinitesimal theory which satisfies modern requirements of rigour and
which is no weaker than Weierstrass's theory.  This reconstruction
makes infinitesimal theory an almost respectable ancestor of a fully
fledged, powerful modern theory, lifts it from the status of
pre-scientific gibberish, and renews interest in its partly forgotten,
partly falsified history.
\end{quote}
Not everyone was persuaded.  A decade later, Courant's {\em duty of
avoiding every such hazy idea\/} was taken up under a constructivist
banner by E.~Bishop.  In his essay \cite{Bi75} cast in the form of an
imaginary dialog between Brouwer and Hilbert, E. Bishop anchors his
foundational stance in a species of mathematical constructivism.
Thus, Bishop's opposition to Robinson's infinitesimals, expressed in a
bristling review~\cite{Bi77} of Keisler's textbook~\cite{Ke}, was to
be expected (and in fact was anticipated by editor Halmos).  In a
memorable parenthetical remark, Bishop writes \cite{Bi77}:
\begin{quote}
Although it seems to be futile, I always tell my calculus students
that mathematics is not esoteric: It is common sense. (Even the
notorious~$(\epsilon,\delta)$-definition of limit is {\bf common
sense} [emphasis added--MK], and moreover it is central to the
important practical problems of approximation and estimation.)
\end{quote}
Bishop is referring, of course, to the type of {\em common-sense\/}
definition reproduced in~\eqref{disaster}, which he favors over
Keisler's definition \eqref{102b} in terms of Robinson's hyperreals.

Bishop expressed his views about non-standard analysis and its use in
teaching in a brief paragraph toward the end of his essay ``Crisis in
contemporary mathematics" \cite[p.~513-514]{Bi75}.  After discussing
Hilbert's formalist program he writes:

\begin{quote}
A more recent attempt at mathematics by formal finesse is non-standard
analysis. I gather that it has met with some degree of success,
whether at the expense of giving significantly less meaningful proofs
I do not know. My interest in non-standard analysis is that attempts
are being made to introduce it into calculus courses. It is difficult
to believe that {\bf debasement of meaning\/} [emphasis added--MK]
could be carried so far.
\end{quote}
Bishop's view of the introduction of non-standard analysis in the
classroom as no less than a ``debasement of meaning", was duly noted
by J.~Dauben~\cite{Gi92}.

To illlustrate how Bishop anchors his foundational stance in a species
of mathematical constructivism, note that in \cite[pp.~507-508]{Bi75},
he writes:
\begin{quote}
To my mind, it is a major defect of our profession that we refuse to
distinguish [...] between integers that are computable and those that
are not [...] the distinction between computable and non-computable,
or constructive and non-constructive is the source of the most famous
disputes in the philosophy of mathematics...
\end{quote}
On page 511, Bishop defines a principle (LPO) as the statement that 

\begin{quote}
it is possible to search ``a sequence of integers to see whether they
all vanish'',
\end{quote}
and goes on to characterize the dependence on the LPO as a procedure
both Brouwer and Bishop himself reject. S.~Feferman \cite{Fe} explains
the principle as follows:
\begin{quote}
Bishop criticized both non-constructive classical mathematics and
intuitionism.  He called non-constructive mathematics ``a scandal",
particularly because of its ``deficiency in numerical meaning".  What
he simply meant was that if you say something exists you ought to be
able to produce it, and if you say there is a function which does
something on the natural numbers then you ought to be able to produce
a machine which calculates it out at each number.
\end{quote}
The constructivist objections to LPO are similar to objections to the
law of excluded middle, and to proofs by contradiction (for example,
the traditional argument for the irrationality of $\sqrt{2}$ is a
proof by contradiction, but using classical continued fraction
estimates, it is easy to rewrite it in a numerically meaningful
fashion, acceptable from the constructivist viewpoint).

Given that a typical construction of Robinson's infinitesimals (see
Keisler \cite[p. 911]{Ke}) certainly does rely on LPO, Bishop's
opposition to such infinitesimals, expressed in a bristling review
\cite{Bi77} of Keisler's textbook, was to be expected.


Non-standard calculus in the classroom was analyzed in the Chicago
study by K.~Sullivan \cite{Su}.  Sullivan showed that students
following the non-standard calculus course were better able to
interpret the sense of the mathematical formalism of calculus than a
control group following a standard syllabus.  Such a conclusion was
also noted by M.~Artigue~\cite[p.~172]{Ar}.

\section{A 10-step proposal}
\label{proposal}

In the matter of teaching decimal notation, we would like to obtain
some reaction from educators to the following proposal concerning the
problem of the unital evaluation of~$.999\ldots$.  A student may ask:
\begin{quote}
What does the teacher mean to happen exactly after {\em nine, nine,
nine\/} when he writes {\em dot, dot, dot}?
\end{quote}

How is a teacher to handle such a question?  Experience shows that
toeing the standard line on the unital evaluation of~$.999\dots$
possesses a high-frustration factor in the classroom.  Rather than
baffling the student with such a categorical claim, a teacher could
proceed by presenting the following ten points, based on the material
outlined in Section~\ref{ten}:

\begin{enumerate}
\item
the reals are not, as the rationals are not, the maximal number
system;
\item
there exist larger number systems, containing infinitesimals;
\item
in such larger systems, the interval~$[0,1]$ contains many numbers
infinitely close to~$1$;
\item
in a particular larger system called the hyperreal numbers, there is a
generalized notion of decimal expansion for such numbers, starting in
each case with an unbounded number of digits~``$9$'';
\item
all such numbers therefore have an arguable claim to the
notation~``$.999\ldots$'' which is patently ambiguous (the meaning of
the ellipsis~``$\ldots$'' requires disambiguation);
\item
all but {\em one\/} of them are strictly smaller than~$1$;
\item
the {\em convention\/} adopted by most professional mathematicians is
to interpret the symbol ``$.999\ldots$'' as referring to the {\em
largest\/} such number, namely~$1$ itself;
\item
thus, the students' intuition that~$.999\ldots$ falls just short
of~$1$ can be justified in a mathematically rigorous fashion;
\item
the said extended number system is mostly relevant in infinitesimal
calculus (also known as differential and integral calculus);
\item
if you would like to learn more about the hyperreals, go to your
teacher so he can give you further references.
\end{enumerate}

\section{Epilogue}

A goal of our, admittedly non-standard, analysis is both to educate
and to heal.  The latter part involves placing balm upon the
bewilderment of myriad students of decimal notation, frustrated by the
reluctance of their education professionals to yield as much as an
infinitesimal iota in their evaluation of the symbol~$.999...$, or to
acknowledge the ambiguity of an ellipsis.

\section*{Acknowledgments}

We are grateful to David Ebin and David Tall for a careful reading of
an earlier version of the manuscript, and for making numerous helpful
suggestions.

\end{document}